\documentclass[draft]{article}

\usepackage[leqno,fleqn]{amsmath}
\usepackage{amsthm}
\usepackage{amsbsy}
\usepackage{comment}
\usepackage{enumerate}
\usepackage[normalem]{ulem}
\usepackage{tikz}
\usetikzlibrary{positioning,patterns,calc}

\usepackage{amsfonts}
\usepackage{mathrsfs}
\usepackage{MnSymbol}
\DeclareMathAlphabet{\mathpzc}{OT1}{pzc}{m}{it}

\newcommand{\duline}[1]{{\bgroup \markoverwith{{\bgroup \markoverwith{\rule[-1.2pt]{0.1pt}{0.4pt}}\ULon {\rule[-2.8pt]{1pt}{0.4pt}}}}\ULon {#1}}} % define to insert two lines approprietly under M
\newcommand{\dulineF}[1]{{\bgroup \markoverwith{{\bgroup \markoverwith{\rule[-1.2pt]{0.4pt}{0.4pt}}\ULon {\rule[-2.8pt]{2pt}{0.4pt}}}}\ULon {#1}}} % define to insert two lines approprietly under F
\newcommand{\dulineeta}[1]{{\bgroup \markoverwith{{\bgroup \markoverwith{\rule[0.2pt]{0.1pt}{0.4pt}}\ULon {\rule[-1.4pt]{1pt}{0.4pt}}}}\ULon {#1}}} % define to insert two lines approprietly under M
\newcommand{\real}{\mathbb{R}} % real numbers
\newcommand{\integer}{\mathbb{Z}} % integers
\newcommand{\naturals}{\mathbb{N}} % naturals
\newcommand{\supp}{\textup{supp}} % support
\newcommand{\rest}[1]{\,\rule[-.17cm]{.012cm}{.3cm}_{\,#1}} % restriction of a function or diff at a point (a vertical line slightly displaced downward)
\renewcommand{\setminus}{\,\backslash\,}
\newcommand{\Index}{\text{index}}
\newcommand{\ricci}{\text{Ric}}
\newcommand{\cardinality}{\#}
\newcommand{\weakbound}[1]{\sup\mathcal{I}(#1)+\cardinality\mathcal{A}(#1)}
\newcommand{\strongbound}[1]{\sup\mathcal{I}(#1)+\sup\mathcal{A}(#1)}

\newcommand{\area}{\textup{area}} % area/volume of a submanifold
\newcommand{\lwidth}{\textbf{L}} % the width of a sequence of discrete maps
\newcommand{\critical}{\textbf{C}} % critical set of a sequence of discrete maps

\newcommand{\mass}{\textbf{\duline{M}}} % mass of a current
\newcommand{\fmetric}{\textbf{\dulineF{F}}} % F-metric for currents and varifolds
\newcommand{\flatnorm}{\mathcal{F}} % flat norm for currents
\newcommand{\varifolds}{\mathcal{V}} % set of varifolds for manifolds (and R^n I think)
\newcommand{\rectifiable}{\textup{R}\mathcal{V}} % rectifiable varifolds for manifolds
\newcommand{\integral}{\textup{I}\mathcal{V}} % integral varifolds for manifolds
\newcommand{\cycles}{\mathcal{Z}} % closed currents
\newcommand{\currents}{\mathcal{I}} % currents
\newcommand{\ncycles}{\mathcal{Z}_n(M;\integer_2)} % n-cycles mod 2 in M
\newcommand{\ncyclestop}[1]{\mathcal{Z}_n(M;#1;\integer_2)} % closed n-cycles mod 2 in M with some topology
\newcommand{\eminimising}{\mathfrak{A}}

\newcommand{\width}{\omega} % width
\newcommand{\mwidth}[1]{\width^{(#1)}}
\newcommand{\sweepout}{\mathcal{P}} % set of weepouts

\newcommand{\minimal}{\Lambda}
\newcommand{\minimalcycles}{\mathcal{T}}

\newcommand{\catone}{\mathcal{N}_1\textup{-cat}}
\newcommand{\Acat}{\mathcal{A}\textup{-cat}}

\newcommand{\lusternik}{\mathcal{LS}}

\theoremstyle{plain}
\newtheorem{theorem}{Theorem}[section]
\newtheorem{proposition}[theorem]{Proposition}
\newtheorem{lemma}[theorem]{Lemma}
\newtheorem{corollary}[theorem]{Corollary}
\newtheorem{claim}{\texttt{Claim}}
\makeatletter
\@addtoreset{claim}{theorem}
\makeatother

\theoremstyle{definition}
\newtheorem{definition}[theorem]{Definition}

\theoremstyle{remark}
\newtheorem{remark}{\underline{Remark}}
\numberwithin{equation}{theorem}

\begin{document}

%title
\title{Non-compactness of the space of minimal hypersurfaces}
\author{Nicolau Sarquis Aiex
%\renewcommand{\thefootnote}{\arabic{footnote}}
%\footnote{Department of Mathematics, Imperial College London, $180$ Queen’s Gate, Huxley Building, SW7 2AZ, London, United Kingdom - \texttt{n.sarquis12@imperial.ac.uk}}
%$^{,}$
\footnote{The author was supported by a CNPq-Brasil Scholarship}
}
%\date{}
%\thanks{The author was supported by a CNPq-Brasil Scholarship}
\maketitle

% abstract
\renewcommand{\abstractname}{\vspace{-\baselineskip}}
\begin{abstract}
\noindent \textsc{Abstract.} We show that the space of min-max minimal hypersurfaces is non-compact when the manifold has an analytic metric of positive Ricci curvature and dimension $3\leq n+1\leq 7$. Furthermore, we show that bumpy metrics with positive Ricci curvature admit minimal hypersurfaces with unbounded $\textup{index}+\textup{area}$. When combined with the recent work fo F.C. Marques and A. Neves, we then deduce some new properties regarding the infinitely many minimal hypersurfaces they found.
\end{abstract}

%\noindent \textbf{Mathematics Subject Classification} $49$Q$05$, $53$A$10$.

% introduction
\section{Introduction}

In \cite{lusternik-schnirelmann}, L. Lusternik and L. Schnirelmann defined the category of a manifold $M$ as the least number of contractible open sets that cover it, which we denote by $\lusternik(M)$.
They also showed that any smooth function defined on $M$ has at least $\lusternik(M)$ critical points (see for example \cite{ocorneaetal1}).
Their proof consists in finding a non-decreasing sequence of critical values $c_1\leq\ldots\leq c_{\lusternik(M)}$ using a standard min-max approach.
Then, it divides in two cases: either the sequence is strictly increasing or $c_k=c_{k+i}$ for some $k\in\{1,\ldots,\lusternik(M) \}$ and $i\in\naturals$.
In the first case the proof is finished and in the latter they show that the category of the critical level set has to be greater or equal to $i+1$, hence it cannot be a finite set of points.

In this paper we are interested in using their ideas to prove the same result of the second case on the weak setting of the mass functional defined on the space of $n$-cycles mod $2$ on a $(n+1)$-dimensional manifold.
The space of flat cycles does not have a smooth structure and the mass functional is not even continuous on the flat topology so it is not possible to carry over the methods directly.
To be more precise, we study the $p$-width $\{\width_p\}_{p\in\naturals}$ and show that whenever the Riemannian metric has positive curvature and is analytic, then $\width_p=\width_{p+N}$ implies that the level set of minimal hypersurfaces of area $\width_p$ has category greater or equal to $N+1$.

The ideas if Lusternik and Schnirelmann were already successfully applied in this scenario by F.C. Marques and A. Neves \cite{fmarques-aneves2} in which they show the existence of infinitely many minimal hypersurfaces on manifolds of positive Ricci curvature.
Together with their outstanding result and a compactness theorem by B. Sharp \cite{bsharp1} we are able to further show that the space of minimal hypersurfaces is non-compact when the Riemannian metric is analytic as well as having positive Ricci curvature.
Similarly to the work of Marques-Neves, we make use of the Gromov-Guth \cite{lguth1} growth estimates for the width.
As a consequence we show the existence of infinitely many non-congruent minimal embeddings on analytic perturbations of the round sphere $S^n$, with positive Ricci curvature, for $4\leq n\leq 7$.

This work is divided as follows. 
In section 2 we establish notation and cover some preliminaries to make it sufficiently self-contained. 
All of the results and definitions in this section are taken from \cite{fmarques-aneves2}.
In section 3 we introduce the concept of $\mathcal{A}$-category and we prove the topological theorem about the critical set under the equality case.
In section 4 we apply the result of the previous section to some specific cases and we prove the main non-compactness result.

%acknowledgements
\hfill

\textit{
Acknowledgements: I am thankful to my PhD adviser Andr\'e Neves for his guidance and suggestion to work on this problem. I would like to thank the comments of Alessandro Carlotto and Fernando Cod\'a Marques as well as Ben Sharp for helpful discussions and several corrections.
}

% preliminaries
\section{Preliminaries}
Throughout this section we assume that $(M,g)$ is a Riemannian manifold of dimension $n+1$ ismetrically embedded in $\real^N$ for some $N\in\naturals$.
We will establish notations and definitions that are not standard in the literature.

\hfill

\noindent\textbf{Varifolds and Currents}

Denote by $\currents_k(M;\integer_2)$ and $\cycles_k(M;\integer_2)$ the spaces of $k$-currents modulo $2$ and $k$-cycles in $M$, respectively. 
Let $\rectifiable_k(M)$ be the space of $k$-dimensional rectifiable varifolds in $\real^N$ whose support lies in $M$ with the weak topology (we agree with the definition in \cite[\textsection 2]{jpitts1}).
The subspace of $k$-dimensional integral varifolds is denoted by $\integral_k(M)\subset\rectifiable_k(M)$.

Given $V\in\rectifiable_k(M)$ we denote by $\|V\|$ the Radon measure in $M$ associated with $V$, we call $\|V\|(M)$ the mass of $V$. 
Now, given a $k$-current $T\in\currents_k(M;\integer_2)$ we denote $|T|\in\integral_k(M)$ the integral varifold associated to $T$ and to simplify notation we write $\|T\|$ its associated Radon measure in $M$.
Reversely, if $V\in\integral_k(M)$ then $[V]\in\currents_k(M;\integer_2)$ denotes the unique $k$-current satisfying $\Theta^k([V],x)=\Theta^k(V,x)\; mod\;2$ for all $x\in M$ (see \cite{bwhite1}).

The weak topology in $\rectifiable_k(M)$ is induced by the $\fmetric$-metric, denoted by $\fmetric$ (see \cite[\textsection 2]{jpitts1}).
On the space of currents we will work with three different topologies induced by the flat metric $\flatnorm$, the mass $\mass$ and the $\fmetric$-metric for currents also denoted by $\fmetric$.
For the definition of the first two see \cite[\textsection 4.2.26]{hfederer1}, the latter is defined as
\begin{equation*}
   \fmetric(T,S)=\flatnorm(T-S) + \fmetric(|T|,|S|), 
\end{equation*}
for all $T,S\in\currents_k(M;\integer_2)$.
We will always assume $\currents_k(M;\integer_2)$ and $\cycles_k(M;\integer_2)$ to be endowed with the flat topology unless otherwise specified.

\hfill

\noindent\textbf{Almost-minimising Varifolds}

For our purposes it will be sufficient to only consider $\integer_2$-almost-minimising varifolds, the definition is the same for a different group $G$ (see \cite[\textsection 3.1]{jpitts1}).
We also remark that our definition is slightly different from \cite{jpitts1} but all the results therein contained remain true.

\begin{definition}
Let $U\subset M$ be an open set, $\varepsilon>0$ and $\delta>0$. We define
\begin{equation*}
\eminimising_k(U; \varepsilon,\delta)\subset\cycles_k(M;\integer_2)
\end{equation*}
to be the set of cycles $T\in\cycles_k(M;\integer_2)$ such that any finite sequence $T_1,\ldots,T_m \in \cycles_k(M;\integer_2)$ satisfying
\begin{enumerate}[(a)]
   \item $\supp(T-T_i)\subset U$ for all $i=1,\ldots, m$;
   \item $\flatnorm(T_i,T_{i-1})\leq\delta$ for all $i=1,\ldots, m$ and
   \item $\mass(T_i)\leq\mass(T)+\delta$
\end{enumerate}
must also satisfy
\begin{equation*}
\mass(T_m)\geq\mass(T)-\varepsilon.
\end{equation*}

We say that a varifold $V\in\varifolds_k(M)$ is almost-minimising in $U$ if for every $\varepsilon>0$ there exists $\delta>0$ and $T\in\eminimising_k(U;\varepsilon;\delta)$ such that
\begin{equation*}
   \fmetric(V,|T|)<\varepsilon.
\end{equation*}

Furthermore, we say that $V$ is almost-minimising in annuli if for every $p\in\supp\|V\|$ there exists $r>0$ such that $V$ is almost-minimising in the annulus $A(p;s,r)=B(p,r)\setminus B(p,s)$ for all positive $s<r$.
\end{definition}

The following is a well known regularity theorem for stationary varifolds of codimension $1$.
This was originally proven in by Pitts, up to dimension $n+1\leq 6$ and later extended by Schoen-Simon to $n+1\leq 7$.
\begin{theorem}[{\cite[\textsection 7]{jpitts1}, \cite[\textsection 4]{rschoen-lsimon1}}]\label{regularity theorem}
Let $M^{n+1}$ be a closed manifold of dimension $n+1$ with $3\leq n+1 \leq 7$.
If $V\in\integral_n(M)$ is stationary and almost-minimising in annuli, then $\supp\|V\|$ is a smooth embedded minimal hypersurface.
\end{theorem}
\begin{remark}
If $n\geq 7$ then it was also proven that $\supp\|V\|$ has a singular set of Hausdorff dimension at most $n-7$.
\end{remark}
%%%%
%%%%
%%%% almgren-pitts minmax theory
%%%%

\hfill

\noindent\textbf{Almgren-Pitts Min-max Theory}

We want to present the appropriate modification of the Almgren-Pitts Min-max Theory that will be necessary.
All of the results and definitions are taken from \cite{fmarques-aneves1} where one can find detailed proofs.
Henceforth we restrict ourselves to the codimension one case, that is, $k=n$ and $M$ has dimension $n+1$.

Firstly, given a cell complex $X$ and $l\in\naturals\cup\{0\}$ we denote by $X_{(l)}$ the set of $l$-cells. 
Let $I^m=[0,1]^m$ denote the $m$-dimensional cube, for each $j\in\naturals$ we denote by $I(1,j)$ the cell decomposition of $I=I^1$ whose $0$-cells and $1$-cells are given by
\begin{equation*}
   \begin{aligned}
      I(1,j)_{(0)} & = \{[0],[3^{-j}],\ldots,[1-3^{-j}],[1]\}, \\
      I(1,j)_{(1)} & = \{[0,3^{-j}],\ldots,[1-3^{-j},1]\}.
   \end{aligned}
\end{equation*}
Now, if $m>1$ then, for each $j\in\naturals$, the standard cell complex of $I^m$ is defined as 
\begin{equation*}
   I(m,j)=\underbrace{I(1,j) \otimes\ldots\otimes I(1,j)}_{m \text{ times}}.
\end{equation*}
\begin{definition}
A set $X\subset I^m$ is said to be a \textit{cubical subcomplex} of $I^m$ if $X$ is a subcomplex of $I(m,j)$ for some $j\in\naturals$.
%That is, $X_{(l)}\subset I(m,j)_{(l)}$ for all $0\leq l\leq \dim(X)$.
\end{definition}
By abuse of notation we write $X$ for both the cell decomposition and its support.
Note that the dimension of $X$ is not required to be $m$.

If $X$ is a cubical subcomplex of $I(m,j)$ and $l\geq j$ we write $X(l)$ for the union of all cells in $I(m,l)$ whose support is contained in $X$.

\begin{definition}
For a cubical subcomplex $X\subset I^m$, we say that a map $\Phi:X\rightarrow \ncycles$ has \textit{no concentration of mass} if
\begin{equation*}
   \lim_{r\rightarrow 0}\sup\{\|\Phi(x)\|(B(q,r)):x\in X \text{ and } q\in M\}=0.
\end{equation*}
\end{definition}
\begin{remark}
One can show that mass continuous maps have no concentration of mass (see \cite[Lemma 3.8]{fmarques-aneves2})
\end{remark}

\begin{definition}
Let $\{m_i\}\subset \naturals$ be positive integers, $X_i\subset I^{m_i}$ cubical subcomplexes and $S=\{\Phi_i:X_i\rightarrow \ncycles\}$ a sequence of flat continuous maps.
We define the width of a sequence of maps as
\begin{equation*}
   \lwidth(S)=\limsup_{i\rightarrow\infty} \sup\{\mass(\Phi_i(x)):x\in X_i\}
\end{equation*}
and the following compact set of critical varifolds
\begin{equation*}
   \begin{aligned}
      \critical(S)=\{V\in\rectifiable_n(M) :\, & V=\lim_{j\rightarrow\infty}|\Phi_{i_j}(x_j)| \text{ for some increasing sequence }\\
                                               & \quad \{i_j\}_{j\in\naturals} \text{, } x_j\in X_{i_j} \text{ and } \|V\|(M)=\lwidth(S)\}.
   \end{aligned}
\end{equation*}

In case we have a fixed domain $X$ and a map $\Phi:X\rightarrow\ncycles$ it defines an homotopy class (with free boundary) of maps with no concentration of mass
\begin{equation*}
   \begin{aligned}
      \null[\Phi] = \{ \Psi:X\rightarrow\ncycles:\, \Psi & \text{ is flat homotopic to } \Phi\\
                                                    & \text{ and has no concentration of mass} \}
   \end{aligned}
\end{equation*}
and its width is given by
\begin{equation*}
   \lwidth[\Phi]=\inf_{\Psi\in[\Phi]}\sup_{x\in X}\mass(\Psi(x)).
\end{equation*}
\end{definition}
\begin{remark}
Although the nomenclature is the same it will always be clear when we refer to the width of a sequence or the width of an homotopy class.
\end{remark}

The following theorem is a consquence of the interpolation theorems in \cite{fmarques-aneves2}.

\begin{theorem}\label{from flat to f}
Let $X\subset I^m$ be a cubical subcomplex and $\Phi:X\rightarrow\ncycles$ be a flat continuous map with no concentration of mass.
For any $\varepsilon > 0$ there exist $l\in\naturals$, $\tilde{X}=X(l)$ cubical subcomplex and $\tilde{\Phi}:\tilde{X}\rightarrow\cycles_n(M;\mass;\integer_2)$ a mass continuous map satisfying:
\begin{enumerate}[(i)]
   \item $\Phi\rest{\tilde{X}}$ is homotopic to $\tilde{\Phi}$ in the flat topology;
   \item $\lwidth[\Phi] \leq \sup \{\mass(\tilde{\Phi}(x)):x\in\tilde{X} \}\leq \lwidth[\Phi]+\varepsilon $.
\end{enumerate}
\end{theorem}

The critical set of a sequence is the set of candidates to be critical min-max varifolds.
However, it is not even true that they are stationary in general.
We can in fact refine a sequence of flat continuous maps such that its critical set contains only stationary varifolds.
The result follows by applying the previous Theorem to each element of the sequence and a pull-tight procedure.

\begin{theorem}[{\cite{fmarques-aneves1,fmarques-aneves2},\cite[\textsection 4.3]{jpitts1}}]\label{pulltight theorem}
Let $X_i\subset I^{m_i}$ be cubical subcomplexes of $I(m_i,j_i)$ and $S=\{\Phi_i:X_i\rightarrow\ncycles\}$ be a sequence of flat continuous maps with no concentration of mass.
There exist $l_i\geq j_i$, $\tilde{X_i}=X_i(l_i)$ cubical subcomplexes and $\tilde{S}=\{\tilde{\Phi}_i:\tilde{X_i}\rightarrow\cycles_n(M;\mass;\integer_2)\}$ sequence of mass continuous maps such that:
\begin{enumerate}[(i)]
   \item $\Phi_i\rest{\tilde{X_i}}$ is homotopic to $\tilde{\Phi}_i$ in the flat topology;
   \item if $V\in\critical(\tilde{S})$ then $V$ is stationary.
   \item $\lwidth(\tilde{S})\leq \lwidth(S)$;
\end{enumerate}
Furthermore, if $\lwidth(\tilde{S})= \lwidth(S)$ then
\begin{equation*}
   \critical(\tilde{S})\subset \critical(S)\cap\{V\in\rectifiable_n(M):V\text{ is stationary}\}
\end{equation*}
\end{theorem}

The following theorem shows the existence of almost-minimising varifolds and it was originally proven by Pitts for maps with cubical domain and a boundary condition.
However, it remains true for a cubical subcomplex and allowing homotopies with free boundary (see \cite{fmarques-aneves2}).

\begin{theorem}\label{existence of minmax hypersurface}
Let $X\subset I^m$ be a cubical subcomplex and $\Phi:X\rightarrow \cycles_n(M;\fmetric;\integer_2)$ a $\fmetric$-continuous map. If $\lwidth[\Phi]>0$ then there exists $V\in\integral_n(M)$ satisfying
\begin{enumerate}[(i)]
   \item $V$ is stationary;
   \item $V$ is almost-minimising in annuli;
   \item $\|V\|(M)=\lwidth[\Phi]$.
\end{enumerate}
\end{theorem}

From the proof of the previous theorem we extract a result that follows from Pitts' combinatorial arguments \cite[\textsection 4.10]{jpitts1}.
To obtain the version that we state here it is necessary to further apply the interpolation theorems in \cite{fmarques-aneves2}.

\begin{theorem}\label{combinatorial sweepout}
Fix $m\in\naturals$ and let $X_i\subset I^m$ be cubical subcomplexes and $S=\{\Phi_i:X_i\rightarrow\ncycles\}$ be a sequence of flat continuous maps with no concentration of mass such that every $V\in\critical(S)$ is stationary.

If no element of $\critical(S)$ is almost-minimising in annuli then there exist a non-decreasing sequence $\{l_i\}_{i\in\naturals}\subset\naturals$, $X^*_i=X_i(l_i)$ cubical subcomplexes and a sequence of mass continuous maps $S^*=\{\Phi^*_i:X^*_i\rightarrow \cycles_n(M;\mass;\integer_2)\}$ such that
\begin{enumerate}[(i)]
   \item $\Phi_i\rest{X_i^*}$ is homotopic to $\Phi^*_i$ in the flat topology;
   \item $\lwidth(S^*)<\lwidth(S)$.
\end{enumerate}
\end{theorem}

In \cite{falmgren1} F.J. Almgren Jr. shows, in particular, the existence of an isomorphism $\pmb{F}_M:\pi_{q}(\cycles_n(M;\integer_2),\{0\})\rightarrow H_{q+n}(M;\integer_2)$ for all $q\in\naturals$ which is called the Almgren isomorphism.
\begin{definition}\label{sweepout-def}
We say that a flat continuous map $\Phi:S^{1}\rightarrow\cycles_n(M;\integer_2)$ is a \textit{sweepout} if $\pmb{F}_M([\Phi])\neq 0$, where $[\Phi]\in\pi_{1}(\cycles_n(M;\integer_2))$.
\end{definition}

It is possible to show the existence of a fundamental cohomology class $\bar\lambda\in H^{1}(\cycles_n(M;\integer_2);\integer_2)$ such that the $p$-th cup product is non-zero for all $p\in\naturals$, $\bar{\lambda}^p\neq 0$.
In particular it is possible to show that the cohomology ring $H^*(\cycles_n(M;\integer_2);\integer_2)$ is isomorphic to the polynomial ring $\integer_2[\bar\lambda]$ generated by $\bar\lambda\in H^{1}$.
For further details see \cite[\textsection 1]{lguth1}.

\begin{definition}\label{psweepout-def}
Let $X\subset I^m$ be a cubical subcomplex for some $m\in\naturals$, $\Phi:X\rightarrow\cycles_n(M;\integer_2)$ a flat continuous map and $p\in\naturals$.
We say that $\Phi$ is a $p$-\textit{sweepout} if
\begin{equation*}
   \Phi^*({\bar\lambda}^p)\neq 0 \in H^{p}(X;\integer_2),
\end{equation*}
where ${\bar\lambda}^p$ is the $p$-th cup product of $\bar\lambda$.
This is equivalent to saying that there exists $\lambda\in H^1(X;\integer_2)$ such that
\begin{enumerate}[(a)]
   \item given any map $\gamma:S^{1}\rightarrow X$, we have $\lambda(\gamma)\neq 0$ if, and only if, $\Phi\circ\gamma$ is a sweepout (as in Definition \ref{sweepout-def}) and
   \item $\lambda^p\neq 0$ in $H^{p}(X;\integer_2)$.
\end{enumerate}
We denote by $\sweepout_p(M)$ the set of $p$-sweepouts in $M$ with no concentration of mass and its admissible domains:
\begin{equation*}
   \begin{aligned}
      \sweepout_p(M) = \{(\Phi,X) : \, & X\subset I^m  \text{ is a }  \text{cubical subcomplex for some } m\in\naturals\\
                                       & \text{ and } \Phi: X\rightarrow \cycles_n(M;\integer_2) \text{ is a } p\text{-sweepout }\\
                                       & \text{ with no concentration of mass}\}.
   \end{aligned}
\end{equation*}
Given a fixed $m\in\naturals$ we denote $\sweepout^{(m)}_p(M)=\{(\Phi,X)\in\sweepout_p(M):X\subset I^m\}$, that is, the $p$-sweepouts with no concentration of mass whose domain is contained in a cube $I^m$ of fixed dimension.
%When the codimension is one, $k=n$, we simply write $\sweepout_p$ and $\sweepout^{(m)}_p$.
\end{definition}
Note that a nullhomotopic map is not a sweepout.
It is easy to see that $\sweepout_p^{(m)}(M)\subset\sweepout_p^{(m+1)}(M)$ and $\sweepout_p(M)=\cup_{m\in\naturals}\sweepout_p^{(m)}(M)$.

The following is an adaptation of an elementary result and is often referred as Vanishing Lemma (see \cite{lguth1} or \cite[Claim 6.3]{fmarques-aneves2}).

\begin{lemma}[Vanishing Lemma]\label{vanishing lemma}
Let $p,l\in\naturals$, $X,Y\subset I^m$ two cubical subcomplexes and $Z=X\cup Y$. If $\Phi:Z\rightarrow\cycles_n(M;\integer_2)$ is a $(p+l)$-sweepout and $\Phi\rest{Y}$ is \textbf{not} a $l$-sweepout then $\Phi\rest{X}$ must be a $p$-sweepout.
\end{lemma}
\begin{proof}
Take $\lambda\in H^{1}(Z;\integer_2)$ so that condition $(a)$ of Definition \ref{psweepout-def} is satisfied in $Z$ and $\lambda^{p+l}\neq 0$.
Define $\lambda_X=i_X^*\lambda$ and $\lambda_Y=i_Y^*\lambda$, where $i_X,i_Y$ denote the respective inclusion maps onto $Z$. 
Since every $1$-cycle in $X$ or $Y$ is also in $Z$, then condition $(a)$ with respect to $\lambda_X$ and $\lambda_Y$ is satisfied for both spaces.
We can assume that $(\lambda_Y)^l=0$ and we want to prove that $(\lambda_X)^p\neq 0$.

Consider the exact sequence of the pair $(Z,Y)$:
\begin{equation*}
   H^{l}(Z,Y;\integer_2)\xrightarrow{j_Y^*} H^{l}(Z;\integer_2)\xrightarrow{i_Y^*} H^{l}(Y;\integer_2).
\end{equation*}
Because $i_Y^*(\lambda^l)=0$, there exists $\lambda_1\in H^{l}(Z,Y;\integer_2)$ so that $j_Y^*\lambda_1=\lambda^l$.

Now, suppose $(\lambda_X)^p=(i_X^*\lambda)^p=i_X^*(\lambda^p)=0$ and consider the exact sequence for the pair $(Z,X)$:
\begin{equation*}
   H^{p}(Z,X;\integer_2)\xrightarrow{j_X^*} H^{p}(Z;\integer_2)\xrightarrow{i_X^*} H^{p}(X;\integer_2).
\end{equation*}
If we chose $\lambda_2\in H^{p}(Z,X;\integer_2)$ such that $j_X^*\lambda_2=\lambda^p$, then we will have
\begin{equation*}
   j_Y^*\lambda_1\cup j_X^*\lambda_2=\lambda^{p+l}\in H^{(p+l)}(Z;\integer_2).
\end{equation*}
However, $X\cup Y=Z$, hence $H^*(Z,X\cup Y;\integer_2)=0$.
By the definition of cup product on relative cohomology we must have $\lambda_1\cup\lambda_2\in H^{(p+l)}(Z,X\cup Y;\integer_2)$, that is, $\lambda_1\cup\lambda_2=0$(see \cite[\textsection 3.2]{ahatcher1}). 

On the other hand, we have
\begin{equation*}
   \lambda^{p+l}=j_Y^*\lambda_1\cup j_X^*\lambda_2=j_{X\cup Y}^*(\lambda_1\cup\lambda_2)=0,
\end{equation*}
which is a contradiction.
We conclude that $(\lambda_X)^p\neq 0$, hence $\Phi_X$ is a $p$-sweepout.

\end{proof}

\begin{definition}
Given $p\in\naturals$, the $p$-\textit{width} of $(M,g)$ is defined as
\begin{equation*}
   \width_p(M,g)=\inf_{(\Phi,X)\in\sweepout_p(M)}\sup\{\mass(\Phi(x)):x\in X\}.
\end{equation*}
For a fixed $m\in\naturals$ we define the \textit{restricted} $p$-\textit{width} as
\begin{equation*}
   \mwidth{m}_p(M,g)=\inf_{(\Phi,X)\in\sweepout_p^{(m)}(M,g)}\sup\{\mass(\Phi(x)):x\in X\},
\end{equation*}
where we only consider $p$-sweepouts whose domain is contained in a cube $I^m$ of fixed dimension $m$.
%When the codimension is one, $k=n$, we simply write $\width_p$ and $\width^{(m)}_p$.
\end{definition}

\begin{remark}
Note that for any $p$-sweepout $\Phi$ it is true that $\width_p\leq \lwidth[\Phi]$.
However, it is not known in general whether it is always possible to have equality for some sweepout.
%Having such optimal sweepout is a very strong property because it would give us for free a $p$-parameter family of perturbations of the critical varifold.
It is trivial from the definition that we can always find a sequence of $p$-sweepouts $S$ that satisfies $\lwidth(S)=\width_p$.
Nevertheless that is not very useful because we must allow the ambient cubical domain $I^{m_i}$ to vary and in this case Pitts' combinatorial construction does not work (see \cite[\textsection 4.10]{jpitts1}).
\end{remark}

%\begin{remark}\label{index bound remark}
%We would like to state here a result by F.C. Marques and A. Neves that was communicated in conferences but at the moment it has not yet been published \cite{fmarques-aneves3}.
%
%They prove, in particular, that the minimal hypersurface $\Sigma=\supp\|V\|$ given by Theorem \ref{existence of minmax hypersurface} must have $\textup{index}(\Sigma)\leq \dim(X)$.
%\end{remark}

% lusternik-schnirelmann category of critical sets
\section{Category of a Critical Set}
In this section we are going to explain the notion of $\mathcal{A}$-category of a set, which is a generalization of the Lusternik-Schnirelmann Category (see \cite{mclapp-dpuppe1}).
We will use this alternate notion of category to study the topology of the space of min-max minimal hypersurfaces.

Let us briefly explain the reason for not using the classical definition.
We are working with the space $\ncycles$ and we want to obtain a lower bound on the category of the set of minimal hypersurfaces.
Since $\ncycles$ might not be locally contractible this result could be useless as a covering of the critical set by contractible sets might not exist.
%However, it follows from \cite[Theorem 8.2]{falmgren1} that $\ncycles$ is locally $q$-connected for each $q\in\naturals$, which indicates that the $q$-connected Category is the right choice.

\begin{definition}[{\cite[1.1]{mclapp-dpuppe1}}]
Let $X$ be a topological space and $\mathcal{A}$ a non-empty collection of non-empty subsets of $X$.
We say that a subset $U\subset X$ is \textit{deformable} to $\mathcal{A}$ in $X$ if there exists $A\in\mathcal{A}$ and an homotopy $h_t:U\rightarrow X$, $t\in[0,1]$, such that $h_0=\iota_U$ is the inclusion map and $h_1(U)\subset A$.

%More generally, for a continuous maps $f:X\rightarrow Y$ we say that $f\rest{X'}$ \textit{factors through} $A$ if there exist maps $\alpha$, $\beta$ as before such that $f\circ\beta\circ\alpha$ is homotopic to $f\rest{X'}$.
%That is, $X'$ is deformable to $A$ if and only if the inclusion map factors through $A$.

A finite covering $\{U_1,\ldots,U_k\}$ of open sets such that each $U_j$ is deformable to $\mathcal{A}$ in $X$ is called a $\mathcal{A}$\textit{-categorical covering}.
Given a subspace $Y\subset X$ we define the $\mathcal{A}$-category of $Y$ as the smallest cardinality $k$ of such covering and we write $\Acat_X(Y)=k$.
If no such covering exists we put $\Acat_X(Y)=\infty$.
\end{definition}
\begin{remark}
The $\mathcal{A}$-category of a subset $Y\subset X$ is relative to the ambient space $X$.
In general the relative category is different from the intrinsic category (seeing $Y$ as a subset of itself).
%This happens for the Lusternik-Schnirelmann category as well.
\end{remark}

In our case we consider the collection
\begin{equation*}
   \begin{aligned}
      \mathcal{N}_1=\{N\subset \cycles_n(M;\integer_2):& U \text{ is open in the flat topology and } \\
                                                       & (\iota_U)_*:\pi_1(U)\rightarrow\pi_1(\ncycles) \text{ is trivial }\}.
   \end{aligned}
\end{equation*}
It follows from \cite[Theorem 8.2]{falmgren1} that for every neighborhood of $0\in\ncycles$ contains an element $U\in\mathcal{N}_1$ such that $0\in U$.
%For each $k$ we will make use of the $\mathcal{N}_{1}$-category instead of the Lusternik-Schnirelmann version.
\begin{remark}
Since $\pi_{l}(\cycles_n(M;\integer_2))=0$ for all $l>1$, it follows that the induced map $(\iota_U)_*$ is trivial for all $l\in\naturals$ whenever $U\in\mathcal{N}_1$.
\end{remark}

We summarize some trivial properties that we will be necessary for our applications.

\begin{proposition}\label{properties of category}
Let $\mathcal{N}_1$ be defined as above. For any subset $Y\subset \ncycles$ the following holds:
\begin{enumerate}[(i)]
   \item $\catone_{\cycles_n}(Y)=1$ if and only if $Y$ is contained in an open set $U$ such that $(\iota_U)_*:H_1(U)\rightarrow H_1(X)$ is zero;
   \item if $W\subset Y$ then $\catone_{\cycles_n}(W)\leq \catone_{\cycles_n}(Y)$;
   \item if $K\subset \ncycles$ is compact then $\catone_{\cycles_n}(K)<\infty$.
\end{enumerate}
\end{proposition}
\begin{proof}
\textbf{(i)}:It follows from the definition that there must be a set $U$ such that the maps induced by the inclusion on the fundamental group is trivial.
Simply note that the Hurewicz homomorphism is surjective in dimension $1$ and natural, so the induced map in homology must also be trivial.

\noindent \textbf{(ii)} and \textbf{(iii)} are straightforward from the definition and the fact that $\mathcal{N}_1$ defines a local neighborhood system in $\cycles_n(M;\integer_2)$.
\end{proof}

Before proceeding we explain how the $\mathcal{N}_1$-category relates to the Lusternik-Schnirelmann category.

\begin{definition}
   Let $\Lambda$ be a manifold and $\mathcal{P}$ be the collection consisting of a single one-point set.
   We define the Lusternik-Schnirelmann category of $\Lambda$ as
   \begin{equation*}
      \lusternik(\Lambda)=\mathcal{P}\textup{-cat}_{\Lambda}(\Lambda).
   \end{equation*}
\end{definition}

This definition is the same as in \cite{ocorneaetal1} except that we start counting the size of a covering from $1$ instead of $0$ (see also \cite[\textsection 1.2(1)]{mclapp-dpuppe1}).
So, every result holds simply by subtracting $1$ from their definition.

\begin{lemma}\label{contraction lemma}
Let $\Lambda\subset \ncycles$ be homeomorphic to a complete manifold.
If $K\subset \Lambda$ is a compact set in $\Lambda$, then there exists $\varepsilon_0>0$ such that every closed curve $\gamma(S^1)\subset B^{\flatnorm}(K,\varepsilon_0)$ is flat homotopic in $\ncycles$ to a curve $\tilde{\gamma}(S^1)\subset K$.
\end{lemma}

\begin{proof}
Since $\Lambda$ is homeomorphic to a manifold we can define a Riemannian metric in $\Lambda$ whose distance we denote by $d_\Lambda$.
It follows that the flat topology induced by $\ncycles$ and the topology induced by $d_\Lambda$ coincide.
Therefore, given $\delta>0$, by compactness of $K$ we can find $C_0=C_0(K,\delta)$ such that for every $T,S\in K$ we have
\begin{equation*}
   \flatnorm(T-S)<C_0 \Rightarrow d_{\Lambda}(T,S)<\delta.
\end{equation*}
Now, \cite[Theorem 8.2]{falmgren1} implies that there exists $\eta_0=\eta_0(M)>0$ such that any closed curve $\alpha:[0,1]\rightarrow\ncycles$ with $\flatnorm(\alpha(t))<\eta_0$ for all $t\in[0,1]$ is nullhomotopic in $\ncycles$.
Put $\delta_0>0$ such that $3\delta_0<\min\{\textup{convex}(K),\eta_0\}$, where $\textup{convex}(K)>0$ denotes the convexity radius of $K$ defined by the Riemannian metric on $\Lambda$.
Finally, choose $\varepsilon_0>0$ such that $3\varepsilon_0<\min\{C_0,\eta_0\}$.

Take a closed curve $\gamma:[0,1]\rightarrow\ncycles$ such that 
\begin{equation*}
   \flatnorm(\gamma(t),K)=\inf\{\flatnorm(\gamma(t)-S):S\in K\}<\varepsilon_0
\end{equation*}
for all $t\in[0,1]$.
Choose a partition $0=t_0<\ldots<t_n=1$ satisfying
\begin{equation*}
   \flatnorm(\gamma(s)-\gamma(s'))<\varepsilon_0
\end{equation*}
for all $s,s'\in[t_i,t_{i+1}]$ and $s,s'\in[0,t_1]\cup[t_{n-1},1]$ (because $\gamma$ is closed, $\gamma(0)=\gamma(1)$).
For each $i=0,\ldots,n-1$ we can choose a non-unique $T_i\in K$ that realizes the flat distance to $K$ and $T_n=T_0$.
Then,
\begin{equation*}
   \flatnorm(\gamma(t_i)-T_i)=\flatnorm(\gamma(t_i),K) < \varepsilon_0,
\end{equation*}
for $i=0,\ldots,n$.

It follows that $\flatnorm(T_i-T_{i+1})<2\varepsilon_0<C_0$ so that $d_\Lambda(T_i,T_{i+1})<\delta_0<\textup{convex}(K)$ for all $i=0,\ldots,n-1$.
Now, for each $i=0,\ldots,n-1$ we can take a minimising geodesic $\alpha_i:[0,1]\rightarrow \Lambda$ connecting $T_i$ to $T_{i+1}$ such that
\begin{equation*}
   d_{\Lambda}(\gamma(t),K)<d_{\Lambda}(\gamma(t),T_i)<\delta_0
\end{equation*}
for all $t\in[0,1]$.
That is, $\alpha_i([0,1])\subset B_\Lambda(K,\delta_0)$, where $B_\Lambda(K,\delta_0)=\{S\in\Lambda:\,d_\Lambda(S,K)<\delta_0\}$.
In particular, we have for each $i=0,\ldots,n-1$
\begin{equation*}
   \begin{aligned}
      \flatnorm(\alpha_i(t)-\gamma(s)) & \leq \flatnorm(\alpha_i(t)-T_i) + \flatnorm(T_i-\gamma(s))\\
                                       & \leq d_{\Lambda}(\alpha_i(t),T_i) + \flatnorm(T_i-\gamma(t_i)) + \flatnorm(\gamma(t_i)-\gamma(s))\\
                                       & < \delta_0 + \varepsilon_0 + \varepsilon_0\\
                                       & < \eta_0
   \end{aligned}
\end{equation*}
for all $t\in[0,1]$ and $s\in[t_i,t_{i+1}]$.

At last, we define $\hat\gamma:[0,1]\rightarrow B_{\Lambda}(K,\delta_0)$ as 
\begin{equation*}
   \gamma(s)=\alpha_i\left( \frac{s-t_i}{t_{i+1}-t_i} \right),\text{ if } s\in[t_i,t_{i+1}].
\end{equation*}
By our construction it follows that $\hat \gamma$ is a closed curve and
\begin{equation*}
   \flatnorm(\hat\gamma(s)-\gamma(s))<\eta_0
\end{equation*}
for all $s\in[0,1]$.
Then, $\hat\gamma-\gamma$ is a nullhomotopic curve in $\ncycles$.
In other words, $\hat\gamma$ is flat homotopic to $\gamma$.

To finish the proof we note that $B_{\Lambda}(K,\delta_0)$ is contained in a finite union of convex balls covering $K$ so we can use the unique minimising geodesics to retract $\hat\gamma$ to a closed curve $\tilde\gamma$ inside $K$.
\end{proof}

We now explain the relation between the Lusternik-Schnirelmann category and $\mathcal{N}_1$-category.
This will follow directly from the previous lemma.

\begin{proposition}
Let $\Lambda\subset \ncycles$ be homeomorphic to a closed manifold.
Then $\lusternik(\Lambda)\geq\catone_{\cycles_n}(\Lambda)$.
\end{proposition}
\begin{proof}
First we point out that $\Lambda$ is, in particular, a normal Absolute Neighbourhood Retract (simply embed it into $\real^N$ for large $N$ and taking a tubular neighbourhood).
So, it follows from \cite[Proposition 1.10]{ocorneaetal1} that the definition of Lusternik-Schnirelmann category by closed sets and open sets coincide.

Suppose $\lusternik(\Lambda)=l$, as we mentioned above, we can take $K_1,\ldots,K_l\subset\Lambda$ closed sets all of which is contractible in $\Lambda$ and $\Lambda= K_1\cup\ldots\cup K_l$.
By assumption $\Lambda$ is closed so each $K_i$ is compact.

Let $U_i=B^{\flatnorm}(K_i,\varepsilon_i)$ be given by Lemma \ref{contraction lemma}.
Then, any closed curve in $U_i$ is flat homotopic in $\ncycles$ to a closed curve in $K_i$.
Since $K_i$ is contractible, we conclude that every closed curve in $U_i$  is nullhomotopic in $\ncycles$, that is, $U_i\in\mathcal{N}_1$.
In other words, $\{U_1,\ldots,U_l\}$ define a $\mathcal{N}_1$-categorical covering, thus $\catone_{\cycles_n}(\Lambda)\leq l$.
\end{proof}

One cannot prove the reversed inequality because the homotopies above mentioned might leave the set $\Lambda$ and a $\lusternik$-categorical covering is required to be contractible in $\Lambda$.
As an example, consider $\Lambda$ a two-point set.
Its intrinsic Lusternik-Schnirelmann category is $2$ but any inclusion in $\ncycles$ would give $\mathcal{N}_1$-category $1$.

The motivation for our main result in this section is to try to obtain a result similar to \cite[2.3(iii)]{mclapp-dpuppe1} in our weaker setting, where we don't have Banach manifolds or a smooth functional.
One could hope to mimic their proof but it is not clear that the critical values $c_i$, defined in their paper, correspond to the width $\width_i$.
It might be possible to show that $c_i$ corresponds to a critical value even in our setting, but even so, nothing is known about its asymptotic behavior, which is a crucial property of $\width_i$.
Nevertheless, we have found that it is possible to obtain information about the topology of the critical set measured by its $\mathcal{N}_1$-category.
To do so we must know how the existence of sweepouts contribute to the $\mathcal{N}_1$-category of a set.
The main property that establishes this relation is given by the next lemma.

\begin{lemma}\label{category sweepout}
Let $K\subset\cycles_n(M;\integer_2)$ be a closed set with $\catone(K)\leq N$. There exists an open set $U\subset \cycles_n(M;\integer_2)$ with $K\subset U$ satisfying the following property:

If $X$ is a cubical subcomplex and $\Phi:X\rightarrow \cycles_n(M;\integer_2)$ is a flat continuous map with $\Phi(X)\subset U$ then $\Phi$ is not a $N$-sweepout.
\end{lemma}
\begin{proof}
We prove it by induction.
If $N=1$ then $K$ is contained in an open set $U\subset\cycles_n(M;\integer_2)$ such that every map $f:S^{1}\rightarrow U$ is nullhomotopic in $\cycles_n(M;\integer_2)$.
So it cannot be a $1$-sweepout.

Assume the result is valid for $N-1$ and suppose $\catone(K)\leq N$.
There exists $U_1,\ldots,U_N$ each of which does not contain $1$-sweepouts and $K\subset U_1\cup\ldots\cup U_N$.
It is clear that $K'=K\setminus U_N$ has $\catone(K')\leq N-1$ so we can take $U'$ with $K'\subset U'$ that doesn't contain $(N-1)$-sweepouts.
We can also assume that $\overline{U'}\subset U_1\cup\ldots\cup U_{N-1}$.
Let $U=U'\cup U'_N$, where $U'_N$ is such that $K\setminus U'\subset U'_N$ and $\overline{U'_N}\subset U_N$.

Now, for $\Phi:X\rightarrow U$ let $X_1=\overline{\{x\in X:\Phi(x)\in U'\}}$ and $X_2=\overline{X\setminus X_1}$. Note that if either $X_1$ or $X_2$ are empty then the result follows.
By the induction hypothesis $\Phi\rest{X_1}$ is not a $(N-1)$-sweepout and, as in the first step, $\Phi\rest{X_2}$ is not a $1$-sweepout.
Thus the Vanishing Lemma \ref{vanishing lemma} implies that $\Phi$ cannot be a $N$-sweepout.
\end{proof}

%Henceforth we shall restrict ourselves to the codimension one case, $k=n$.
Let us denote the set of min-max minimal hypersurfaces as
\begin{equation*}
   \begin{aligned}
      \minimal(M,g) = \{V\in\integral_n(M):\, & \supp \|V\| \text{ is a smooth embedded minimal}\\ 
                                              & \text{hypersurface and } V\in\critical(S) \text{ for some sequence of flat}\\
                                              & \text{continuous maps with no concentration of mass}\}
   \end{aligned}
\end{equation*}
and its associated cycles
\begin{equation*}
   \begin{aligned}
      \minimalcycles = \{T\in\ncycles:\, & \supp T \text{ is a smooth embedded minimal}\\
                                                      & \text{ hypersurface or } T=0 \}.
   \end{aligned}
\end{equation*}
For $\beta>0$ we denote $\minimal_\beta=\{V\in\minimal:\,\|V\|(M)\leq \beta\}$ and similarly $\minimalcycles_\beta=\{T\in\minimalcycles:\,\mass(T)\leq \beta\}$

The next Lemma is a direct application of the Constancy Theorem and lower semicontinuity of the mass (see \cite[Claim 6.2]{fmarques-aneves2}).
\begin{lemma}\label{minmax flat approximation}
Fix $\beta>0$, for every open set $U\subset\ncycles$, with $\minimalcycles_\beta\subset U$, there exists $\varepsilon_0>0$ such that for any $T\in\ncycles$
\begin{equation*}
   \fmetric(|T|,\minimal_\beta)<\varepsilon_0 \, \Rightarrow\, T\in U.
\end{equation*}
\end{lemma}

We are now ready to prove the main theorem of this section.
The proof follows the exact same ideas of \cite[Theorem 6.1]{fmarques-aneves2} with the appropriate modifications.

\begin{theorem}\label{category of critical set}
Let $(M^{n+1},g)$ be a closed Riemannian manifold of dimension $n+1$, with $3\leq n+1\leq 7$, and $m,p,N\in\naturals$ such that $p+N\leq m$.
If $\mwidth{m}_p=\mwidth{m}_{p+N}$ then $\catone(\minimalcycles_{\mwidth{m}_{p+N}})\geq N+1$.
\end{theorem}
\begin{proof}
To simplify notation, put $\omega=\mwidth{m}_p=\mwidth{m}_{p+N}$.

Suppose by contradiction that $\catone(\minimalcycles_\omega)\leq N$. By Lemma \ref{category sweepout} there exists an open set $U\subset\ncycles$ with $\minimalcycles_\omega\subset U$ that does not contain $N$-sweepouts.
It follows from Lemma \ref{minmax flat approximation} that there exists $\varepsilon_0>0$ such that
\begin{equation*}
   \fmetric(|T|,\minimal_\omega)<2\varepsilon_0 \, \Rightarrow\, T\in U.
\end{equation*}

Let $S=\{\Phi_i:X_i\rightarrow\ncycles\}_{i\in\naturals}$, with $X_i\subset I^m$ cubical subcomplexes, be a sequence of $(p+N)$-sweepouts with no concentration of mass such that $\lwidth(S)=\omega$.
By Theorem \ref{pulltight theorem} there exist $X'_i\subset X_i$ cubical subcomplexes and a sequence of mass continuous (in particular $\fmetric$-continuous) $(p+N)$-sweepouts $S'=\{\Phi'_i:X'_i\rightarrow \ncyclestop{\mass}\}$ such that $\lwidth(S')\leq \lwidth(S)$.

We claim that $\lwidth(S')= \lwidth(S)$.
Indeed, if we had $\lwidth(S')< \lwidth(S)$ then for $i$ sufficiently large $\Phi'_i$ would be a $(p+N)$-sweepout such that $\sup\{\mass(\Phi'_i(x')):\,x'\in X'_i\}<\lwidth(S)=\omega$, which is a contradiction.

For each $i\in\naturals$ define $Y_i$ to be the cubical subcomplex of $I^m$ consisting of all cells $\alpha\subset X'_i$ such that
\begin{equation*}
   \sup\{\fmetric(|\Phi'_i(x')|,\minimal_\omega):\, x'\in\alpha\}\geq \varepsilon_0.
\end{equation*}
It follows that $\fmetric(|\Phi'_i(x')|,\minimal^{(m)}_\omega)< 2\varepsilon_0$ for all $x'\in \overline{X'_i\setminus Y_i}$, that is,
\begin{equation*}
   \Phi'_i(\overline{X'_i\setminus Y_i})\subset U.
\end{equation*}
Hence $\Phi'_i\rest{\overline{X'_i\setminus Y_i}}$ is not a $N$-sweepout.
Since $\Phi'_i$ is a $(N+p)$-sweepout, it follows that $Y_i$ must be non-empty and from the Vanishing Lemma \ref{vanishing lemma} we get that $\Phi'_i\rest{Y_i}$ is a $p$-sweepout.

Applying Theorem \ref{pulltight theorem} to $\{\Phi'_i\rest{Y_i}\}_{i\in\naturals}$ we obtain $\tilde{Y}_i\subset Y_i$ cubical subcomplexes and a sequence of mass continuous $p$-sweepouts $\tilde{S}=\{\tilde{\Phi}_i:\tilde{Y}_i\rightarrow \ncyclestop{\mass}\}_{i\in\naturals}$ with $\lwidth(\tilde{S})\leq\lwidth(\{\Phi'_i\rest{Y_i}\}_{i\in\naturals})$.
Note that $\lwidth(\{\Phi'_i\rest{Y_i}\}_{i\in\naturals})\leq\omega$ so we can argue as above to conclude that $\lwidth(\tilde{S})=\lwidth(\{\Phi'_i\rest{Y_i}\}_{i\in\naturals})$.

Thus, $\lwidth(\tilde{S})=\omega$ and $\critical(\tilde{S})\subset \critical(\{\Phi'_i\rest{Y_i}\}_{i\in\naturals})\cap\{V\in\rectifiable_n(M):V\text{ is stationary}\}$.
It follows that $\fmetric(V,\minimal^{(m)}_\omega)\geq\varepsilon_0$ for all $V\in\critical(\tilde{S})$.
In particular no element of $\critical(\tilde{S})$ has smooth embedded support.

By Schoen-Simon's Regularity Theorem \ref{regularity theorem} it follows that no element of $\critical(\tilde{S})$ is almost minimising in annuli.
Applying Theorem \ref{combinatorial sweepout} we obtain a sequence of $p$-sweepouts $S^*$ such that $\lwidth(S^*)<\lwidth(\tilde{S})=\mwidth{m}_p$ which contradicts our initial hypothesis and concludes the proof.
\end{proof}

% applications
\section{Applications}
We will use the result in the previous section together with Sharp's Compactness Theorem \cite[Theorem 2.3]{bsharp1} to derive a non-compactness theorem for the space of all minimal hypersurfaces in a manifold with positive Ricci curvature.

Thourghout this section $(M^{n+1},g)$ denotes a closed Riemannian manifold of dimension $3\leq n+1\leq 7$ and $\minimal=\minimal(M,g)$.

Before proceeding, let us first state a characterization of convergence of minimal hypersurfaces.
Given a minimal hypersurface $\Sigma\subset M$ we denote by $L_\Sigma$ the Jacobi operator acting either on smooth functions (when $\Sigma$ is two-sided) or on normal vectorfields (when it is one-sided).
The following is proved in Claims $4-6$ in \cite{bsharp1}.

\begin{proposition}\label{convergence minimal hypersurfaces}
Let $M^{n+1}$ be a closed Riemannian manifold of dimension $3\leq n+1 \leq 7$, $\{\Sigma_i\}_{i\in\naturals}$, $\Sigma_\infty$ be a sequence of minimal embedded smooth hypersurfaces and $\mathcal{S}\subset \Sigma_\infty$ a finite set of points.
Suppose $\Sigma_i\rightarrow \Sigma_\infty$ in the $\mathcal{C}^{\infty}_{\text{loc}}(M\setminus \mathcal{S})$ graphical sense (see \cite{bsharp1}).
We have the following characterization of $\Sigma_\infty$:
\begin{enumerate}[(i)]
   \item if the convergence is \textbf{one-sheeted} then $\mathcal{S}=\emptyset$;
   \item if $\Sigma_\infty$ is \underline{two-sided} then there exists $u\in\mathcal{C}^\infty(\Sigma_\infty)$ such that
      \begin{equation*}
         \left\{
         \begin{aligned}
            & u  \geq 0\\
            & L_{\Sigma_\infty}(u) = 0. 
         \end{aligned}
         \right.
      \end{equation*}
      Furthermore, if the convergence is at least \textbf{two-sheeted} or $\Sigma_i\cap\Sigma_\infty=\emptyset$ for all $i$ sufficiently large then $u>0$ everywhere and $\Sigma_\infty$ is stable.
      In case the convergence is \textbf{one-sheeted} and $\Sigma_i\cap\Sigma_\infty\neq\emptyset$ for all $i$ sufficiently large then we can further conclude that $\Index(\Sigma_\infty)\geq 1$.
   \item if $\Sigma_\infty$ is \underline{one-sided} and the convergence is \textbf{one-sheeted} then, in addition to $(i)$ we have a normal vectorfield $J\in\mathcal{C}^\infty(\Sigma_\infty,T^{\bot}\Sigma_\infty)$ such that
      \begin{equation*}
         \left\{
         \begin{aligned}
            & J  \not\equiv 0\\
            & L_{\Sigma_\infty}(J)  = 0. 
         \end{aligned}
         \right.
      \end{equation*}
      That is, $J$ is a non-trivial Jacobi field.
   \item if $\Sigma_\infty$ is \underline{one-sided} and the convergence is at least \textbf{two-sheeted} then we must have $\lambda_1(L_{\Sigma_\infty})>0$. 
      In addition, there exists $\tilde{\Sigma}_\infty$, a double covering of $\Sigma_{\infty}$, such that $\lambda_1(L_{\tilde{\Sigma}_\infty})=0$.
      That is, $\tilde{\Sigma}_{\infty}$ is a \underline{two-sided} immersed minimal hypersurface with a non-trivial Jacobi field.
\end{enumerate}
\end{proposition}

For $\Omega\subset \minimal$ we define
\begin{equation*}
   \begin{aligned}
      \mathcal {I}(\Omega) = & \{\Index(\supp\|V\|)\in\integer_{\geq 0}:\, V\in \Omega\},\\
      \mathcal {A}(\Omega) = & \{\area(\supp\|V\|)\in\real_{\leq 0}:\, V\in \Omega\}.
   \end{aligned}
\end{equation*}
We know by Sharp's Compactness Theorem that $\strongbound{\Omega}<\infty$ implies that $\Omega$ is compact in the weak topology.
Furthermore, the convergence is as described in Proposition \ref{convergence minimal hypersurfaces}.
%In particular he shows that whenever $M$ does not admits \underline{one-sided} minimal hypersurfaces, for example when $\pi_1(M)=0$, the set $\Omega$ is compact in the $\mathcal{C}^{\infty}_{\text{loc}}(M)$ \textbf{one-sheeted} graphical convergence.
%The same is true if we assume that $\ricci(g)>0$ but for different reasons (in this case $M$ won't have stable minimal hypersurfaces).

Our goal is to prove that the space $\minimal$ is non-compact, then it is sufficient to show that $\strongbound{\minimal}=\infty$.
However, in the general case we are not able to show this.
We managed to overcome this by considering the quantity $\weakbound{\minimal}$ instead.
%This is somehow a $L^1$ version of the other $L^\infty$ quantity.
%We will see how, in some cases, we are able to go from ``$L^1$'' to ``$L^\infty$''.

\begin{theorem}\label{width equals restricted width}
Let $(M^{n+1},g)$ be a closed Riemannian manifold of dimension $3\leq n+1 \leq 7$.

Fix $p\in\naturals$, if $\cardinality\mathcal{A}(\minimal_{\width_p + 1})<\infty$ then there exists $m\in\naturals$ such that
\begin{equation*}
   \mwidth{m}_p(M)=\width_p(M).
\end{equation*}
\end{theorem}
\begin{proof}
Suppose false, that is, we have a strictly decreasing sequence $\{\mwidth{m}_p\}_{m\in\naturals}$ converging from above to $\width_p$.
In particular we obtain a sequence of $p$-sweepouts with no concentration of mass $\{\Phi_m:X_m\subset I^m\rightarrow \ncycles\}_{m\in\naturals}$ satisfying
\begin{equation*}
   \mwidth{m+1}_p\leq\lwidth[\Phi_{m+1}]<\mwidth{m}_p\leq\lwidth[\Phi_{m}].
\end{equation*}
We can further assume that $\lwidth[\Phi_m]<\width_p+1$ for all $m\in\naturals$.

First we apply Theorem \ref{from flat to f} to each $\Phi_m$ and obtain $\tilde{\Phi}_m$ a mass continuous $p$-sweepout.
In particular is is $\fmetric$-continuous and has no concentration of mass.
Now, from Theorem \ref{existence of minmax hypersurface} and the Regularity Theorem \ref{regularity theorem} we obtain a sequence $\{V_m\}_{m\in\naturals}$ of stationary varifolds with smooth embedded support such that $\|V_{m+1}\|(M)<\|V_m\|(M)$.
We can write for each $m$
\begin{equation*}
   V_m=\sum_{i=1}^{l}n_i\cdot \Sigma_i,
\end{equation*}
where $\Sigma_i$ are minimal hypersurfaces such that $\area(\Sigma_i)\in\mathcal{A}(\minimal_{\width_p + 1})$ and $n_i\in\naturals$.
Since $\mathcal{A}(\minimal_{\width_p + 1})$ is finite there are only finitely many possible values of $\|V_m\|(M)=\sum_{i=1}^{l}n_i\area(\Sigma_i)\leq\width_p+1$ for all $m$, which is a contradiction.
\end{proof}

Now we use the result from the previous section to prove a non-compactness theorem.

\begin{theorem}
Let $(M^{n+1},g)$ be a closed Riemannian manifold of dimension $3\leq n+1 \leq 7$ with $\ricci(g)>0$.
If the metric $g$ is analytic then the space $\minimal(M,g)$ of minimal hypersurfaces (with multiplicity) is non-compact.
\end{theorem}
\begin{proof}
Suppose false, then there exists $C>0$ such that $\strongbound{\minimal}<C$.

\begin{claim}
$\cardinality\mathcal{A}(\minimal)<\infty$
\end{claim}
Indeed, if it is not true then there exists a sequence $\{\Sigma_i\}_{i\in\naturals}$ of multiplicity one minimal hypersurfaces with distinct area.
Without loss of generality we can assume that $\area(\Sigma_i)$ is increasing.
By hypothesis we have $\index(\Sigma_i)<C$ for all $i$.
From Sharp's Compactness Theorem we can take a convergent subsequence, still denoted by $\{\Sigma_i\}_{i\in\naturals}$, with limit $\Sigma_\infty$.

We claim that the convergence must \textbf{one-sheeted}, hence smooth everywhere.
In fact, suppose the convergence is \textbf{two-sheeted}, then we can divide in two cases.
That is to say, whether $\Sigma_\infty$ is \underline{one-sided} or \underline{two-sided}.
If $\Sigma_\infty$ is \underline{one-sided}, then we are in case \ref{convergence minimal hypersurfaces}(iv).
In this case there exists an immersed stable minimal hypersurface $\tilde{\Sigma}_\infty$, which is a contradiction because $\ricci(g)>0$.
If $\Sigma_\infty$ is \underline{two-sided}, then we are in case \ref{convergence minimal hypersurfaces}(ii) with \textbf{two-sheeted} convergence, which give us a contradiction for the same reason.
We conclude that the convergence is graphically smooth everywhere.

Now, a Theorem by L.Simon \cite[\textsection 2 Theorem 3]{lsimon2} says that, in a manifold with analytic metric, there exists a $\mathcal{C}^\infty$-neighbourhood of a minimal hypersurface such that any other minimal hypersurface in that neighbourhood has constant area.
Since we have smooth convergence this shows that $\area(\Sigma_i)$ must be constant and equal to $\area(\Sigma_\infty)$ for all $i$ sufficiently large.
This is a contradiction and it finishes the proof of our first claim.

\begin{claim}
There exists a constant $N\in\naturals$ so that $\width_p<\width_{p+N}$ for all $p\in\naturals$.
\end{claim}
Suppose false, then we can find a sequence $\{p_i\}_{i\in\naturals}$ such that
\begin{equation*}
   \width_{p_i}=\width_{p_i+i}.
\end{equation*}
We already know that $\cardinality\mathcal{A}(\minimal)<\infty$, thus Theorem \ref{width equals restricted width} tells us that for each $i$ there exists $m_i\in\naturals$ so that $\mwidth{m_i}_{p_i}=\width_{p_i}$ and $\mwidth{m_i}_{p_i+i}=\width_{p_i}$.
Hence,
\begin{equation*}
   \mwidth{m_i}_{p_i}=\mwidth{m_i}_{p_i+i}.
\end{equation*}
Finally, it follows from Theorem \ref{category of critical set} that $\catone(\minimalcycles_{\width_{p_i+i}})\geq i$.
By the monotonicity property of $\catone$, Proposition \ref{properties of category}(ii), this implies that $\catone(\minimalcycles)=\infty$.
However, we are supposing that $\minimal$ is compact, which implies that so is $\minimalcycles$. 
This is a contradiction because compact sets must have finite $1$-category, thus proving our second claim.

Now, for each $i\in\naturals$ we can find a $(1+iN)$-sweepout $\Phi_i$ such that
\begin{equation*}
   \width_{1+iN}\leq \lwidth[\Phi_i] <\width_{1+(i+1)N}.
\end{equation*}
For each such sweepout we obtain by Theorem \ref{existence of minmax hypersurface} a stationary varifold $V_i$ whose support is smooth and embedded and $\|V_i\|(M)=\lwidth[\Phi_i]$.
By Frankel's Theorem for manifolds with $\ricci(g)>0$ any two minimal hypersurfaces must intersect, then the support of $V_i$ can only have one connected component, that is, $V_i=n_i\cdot \Sigma_i$ where $\Sigma_i$ is a multiplicity one minimal hypersurface and $n_i\in\naturals$.

We already know that $\area(\Sigma_i)$ can only assume finitely many values, from which follows that $\|V_i\|(M)=n_i\area(\Sigma_i)$ must have at least linear growth in $i$.
However, it is known that $\width_p$ has sublinear growth in $p$ (see \cite[Theorem 1]{lguth1} or \cite[Theorems 5.1 and 8.1]{fmarques-aneves2}).
Thus $\width_{1+iN}$ has sublinear growth, which implies that so does $\lwidth[\Phi_i]$.
We arrive to a contradiction and this concludes the proof of the theorem.
\end{proof}

\begin{corollary}
Let $(S^n,g)$ denote the $n$-sphere for $3\leq n\leq 7$ and $g$ be an analytic perturbation of the round metric that preserves positive Ricci curvature.
Then $S^n$ admits infinitely many non-isometric minimal hypersurfaces.
\end{corollary}
\begin{proof}
Since the round metric in $S^n$ is analytic we can apply the previous theorem.
From Sharp's Compactness Theorem it follows that $\strongbound{\minimal(S^n)}=\infty$, so it must contain a sequence of minimal hypersurfaces with either the index going to infinity or the area.
\end{proof}

We remark that the case of the round sphere and dimension $3$ is very well known and a consequence of H.B.Lawson's constructions of minimal surfaces of arbitrary genus in $S^3$.
The case of the round sphere and dimensions $4,5,6$ and $7$ (as well as other dimensions) is a consequence of W.-Y. Hsiang's works \cite{hsiang1,hsiang2}, where the author constructs infinitely many distinct embeddings of minimal hyperspheres.
Furthermore, a similar result was done by A.Carlotto \cite{carlotto} in dimension $4$ for perturbations of the round sphere satifying different hypothesis where the author constructs minimal hypersphere of unbounded index and bounded area.

We can also change the analyticity hypothesis by a bumpy metric.
In this case we can find a sequence of min-max hypersurfaces whose sum of index and area is unbounded.
In dimension $3$ it is already known the existence of such a sequence with unbounded index.

We say that a metric is bumpy if no immersed minimal hypersurface has a non-trivial Jacobi field.
In \cite{bwhite2} B. White showed that bumpy metrics for embedded minimal hypersurfaces are generic and recently the author extended the same result for bumpy metric for immersed minimal hypersurfaces (see \cite{bwhite3}).

To prove this we also have to use a recent result shown by Marques-Neves in \cite{fmarques-aneves3}.
The authors show that for a given $\fmetric$-continuous $k$-sweepout we can always find a varifold that realizes the width of its homotopy class and has $\Index\leq k$.

\begin{theorem}
Let $(M^{n+1},g)$ be a closed Riemannian manifold of dimension $3\leq n+1 \leq 7$ with $\ricci(g)>0$.
If the metric $g$ is bumpy, then
\begin{equation*}
   \width_p<\width_{p+1}
\end{equation*}
and $\strongbound{\minimal(M,g)}=\infty$.
\end{theorem}
\begin{proof}
The proof is very similar to the previous theorem.
First we show that for every $p$ there exists $m$ such that $\width_p=\width^{(m)}_p$.

\begin{claim}
For a fixed $p$, we have $\cardinality\mathcal{A}(\minimal_{\width_p+1})<\infty$
\end{claim}
Suppose it is false.
Arguing exactly as in the previous theorem we obtain a sequence of varifolds $\{V_m\}_{m\in\naturals}\subset\minimal_{\width_p+1}$ such that $\Index(\supp\|V\|)\leq p$ (see \cite[Theorem 1.2]{fmarques-aneves3}).
By Sharp's Compactness Theorem we know that $V_m\rightarrow V_\infty$ for some $V_\infty\in\minimal_{\width_p+1}$ and the convergence is classified by proposition \ref{convergence minimal hypersurfaces}.
Now, in any situation described in \ref{convergence minimal hypersurfaces} it is possible to construct a non-trivial Jacobi field over $\supp\|V_\infty\|$ or its immersed double covering.
In any case, that is a contradiction.

Suppose now that $\width_p=\width_{p+1}$ that is, $\width^{(m)}_p=\width^{(m)}_{p+1}$ for some $m\in\naturals$.
In particular the set $\Omega={V\in\minimal:\|V\|(M)=\width_p}$ is infinite and $\Index(\supp\|V\|)\leq p+1$ for all $V\in\Omega$.
Arguing as before, these varifolds must accumulate on a minimal hypersurface (possibly immersed) with a non-trivial Jacobi field, which is a contradiction.

The remaining statement follows directly from Sharp's Compactness Theorem.
\end{proof}

% bibliography

\end{document}